\documentclass[11pt]{amsart}
\usepackage{amsmath}
\usepackage{cite}
\usepackage{mathrsfs}
\usepackage{amssymb}
\usepackage{hyperref}
\usepackage{microtype}
\newcommand{\Card}{\mathrm{Card}}
\newcommand{\supp}{\operatorname{supp}}
\newcommand{\wt}{\operatorname{wt}}
\numberwithin{equation}{section}
\newtheorem{theorem}[equation]{Theorem}

\newtheorem{corollary}[equation]{Corollary}
\newtheorem{lemma}[equation]{Lemma}
\newtheorem{claim}[equation]{Claim}

\theoremstyle{remark}
\newtheorem{remark}[equation]{Remark}

\theoremstyle{definition}
\newtheorem{definition}[equation]{Definition}

\title{A Generalisation of Diophantine Tuples}
\author{Zijie Gu}

\begin{document}

\maketitle
\begin{abstract}
    This paper investigates a generalised version of Diophantine tuples in finite fields. 
    By applying Shparlinski's method, we obtain power-saving results on the number of such tuples.
\end{abstract}

\section{Introduction}    

    A Diophantine $m$-tuple over a commutative ring $R$ is a subset $\{a_1, \ldots, a_m\}$ of $R$ such that $a_i a_j + 1$ is a square in $R$ for all $i \ne j$. The study of such sets began with Diophantus of Alexandria, who gave the first example of the set $\left\{\frac{1}{16}, \frac{33}{16}, \frac{17}{4}, \frac{105}{16}\right\}$, in which the product of any two distinct elements, increased by one, is a perfect square over $\mathbb{Q}$. In contrast, the finite fields $\mathbb{F}_q$ provide a simpler setting for such problems. Many researchers have recently obtained interesting results on Diophantine tuples in finite fields~\cite{Dujella,zbMATH07697561,Kim2024,2025Yip,Dixit2022,Guloglu2020,Kim2023}.

    Let $\mathbb{F}_q$ be a finite field of characteristic $q>2$. The fundamental questions that naturally arise are:
\begin{itemize}
\item The existence of Diophantine $m$-tuples in $\mathbb{F}_q$
\item The number $N_m(q)$ of Diophantine $m$-tuples in $\mathbb{F}_q$
\end{itemize}

While determining \(N_{m}(q)\) seems to be an elementary problem, it turns out to be non-trivial as more advanced 
gadgets show up already when \(m = 4\): Dujella and Kazalicki (2021)~\cite{Dujella} showed that the number of Diophantine quadruples over $\mathbb{F}_q$ relates to four non-trivial modular forms.  Instead of finding a closed formula for \(N_{m}(q)\), we may give an asymptotic formula when \(q\) is large.

Dujella and Kazalicki (2021)~\cite{Dujella} established the following
asymptotic formula:
\begin{equation}\label{eq:dk}
    N_m(q) = \frac{q^m}{m! 2^{\binom{m}{2}}} + o(q^m),
\end{equation}

Shparlinski (2023)~\cite{zbMATH07697561} improved this to:
\begin{equation}\label{eq:shparlinski}
    N_m(q) = \frac{q^m}{m! 2^{\binom{m}{2}}} + O(q^{m-1}).
\end{equation}

The general advanced machinery for
estimating \(N_{m}(q)\) is the
so-called ``Lang--Weil estimate''.  Using \eqref{eq:dk} and applying
the Lang--Weil estimate naively, one gets
\begin{equation*}
    N_m(q) = \frac{q^m}{m! 2^{\binom{m}{2}}} + O(q^{m-\frac{1}{2}}).
\end{equation*}
Thus \eqref{eq:shparlinski} is a significant improvement over
\eqref{eq:dk}.

    This paper considers a generalisation of Diophantine tuples in finite fields. There has been significant recent progress on simile studies of $f$-Diophantine sets and related variants by several authors, including Kim, Yip, and Yoo~\cite{Kim2025,2025Yip}, Gyarmati~\cite{Gyarmati2001}, Berczes, Dujella, Hajdu, and Tengely~\cite{zbMATH06608234}, Sadek and El-Sissi~\cite{zbMATH06920373}, as well as Hammonds, Kim, Miller, Nigam, Onghai, Saikia, and Sharma~\cite{zbMATH07661280}, among others.

    \begin{definition}[Admissible polynomial]
        A polynomial $f \in \mathbb{F}_q[x]$ is called \emph{admissible} if it satisfies the following conditions:
        \begin{itemize}
            \item $f$ is non-constant and not a perfect square in $\overline{\mathbb{F}_q[x]}$,
            \item For factorization $f(x) = x^r h(x)$ with $h(0) \neq 0$, the polynomial $h(x)$ is not a perfect square in $\overline{\mathbb{F}_q[x]}$.\footnote{This condition is used in Theorem~\ref{thm:admissible-nonsquare}}
        \end{itemize}
    \end{definition}

    Let $d, m$ be integers with $1 \le d < m \le q$.
    For a finite set $[m]=\{1,\dots,m\}$, we denote by $\binom{[m]}{d}$ the family of all $d$-element subsets.
    For an admissible polynomial $f$, we define:

    \begin{definition}
    A \emph{$d$-$f$-Diophantine $m$-tuple} is a set
    $\{a_1,\ldots,a_m\} \subset \mathbb{F}_q$
    such that for every $I = \{i_1,\ldots,i_d\} \in \binom{[m]}{d}$, the element
    $f(a_{i_1} \cdots a_{i_d})$
    is a square in $\mathbb{F}_q$.
    \end{definition}

    In this paper, we study the number of $d$-$f$-Diophantine $m$-tuples in $\mathbb{F}_q$. Since exactly half of the nonzero elements in $\mathbb{F}_q^*$ are squares, the probability that a randomly chosen element from $\mathbb{F}_q^*$ is a square is $\frac{1}{2}$. Therefore, it can be expected that, for a randomly chosen $m$-tuple, the probability that it is a $d$-$f$-Diophantine $m$-tuple is $\frac{1}{2^{\binom{m}{d}}}$. Thus, the expected number of $d$-$f$-Diophantine $m$-tuples in $\mathbb{F}_q$ is $\frac{q^m}{m! \, 2^{\binom{m}{d}}}$.
    
    This leads to the main result of this paper:
    \begin{theorem}\label{mainthe}
        The number $N_f^{\binom{[m]}{d}}(q)$ of $d$-$f$-Diophantine $m$-tuples, where $f$ is an admissible polynomial, satisfies
        \begin{equation*}
            N_f^{\binom{[m]}{d}}(q) = 
            \frac{q^m}{m! \, 2^{\binom{m}{d}}} +
            \begin{cases}
             O\left(q^{m -\frac{1}{2}}\right), & \text{if } \deg(f)\geq 2, \\
             O\left(q^{m -1}\right), & \text{if } \deg(f)=1
            \end{cases}
        \end{equation*}
        Moreover, the implied constant depends on $\deg(f)$, $m$ and $d$ when $\deg(f)\geq 2$.
    \end{theorem}
    
    To get the formula, we use products of multiplicative characters, specifically the quadratic character on $\mathbb{F}_q$, which detects whether an element is a square. The proof mainly uses the Weil bound for character sums. A method of Shparlinski~\cite{zbMATH07697561} is particularly effective for the case when multiple characters are multiplied together. This enables us to improve the error term by a factor of $q^{1/2}$ compared to direct applications of the Weil estimate.

    \subsection*{Notation}
    We use the following notation throughout the paper.
    
    We follow the usual Landau notation $O(\cdot)$ and $o(\cdot)$.  
    
    For a finite set $S$, we write $\Card(S)$ for its cardinality.  
    
    For integers $m\ge 1$ and $d\ge 0$, we denote by $\binom{[m]}{d}$ the family of 
    all $d$-element subsets of $[m]=\{1,\dots,m\}$.
    
    For a vector $\varepsilon=(\varepsilon_i)$, we set
    \[
    \supp(\varepsilon)=\{\,i:\varepsilon_i\neq 0\,\},\qquad
    \wt(\varepsilon)=\Card(\supp(\varepsilon)).
    \]
  
\section{The Weil estimate}
Let us recall one form of the Weil estimate\cite{zbMATH02121181} in its simplest form.

\begin{definition}
A \emph{multiplicative character} on $\mathbb{F}_q$ is a map
\[
\chi : \mathbb{F}_q \to \mathbb{C}
\]
such that
\begin{itemize}
    \item $\chi(0) = 0$,
    \item $\chi(a) \neq 0$ for all $a \in \mathbb{F}_q \setminus \{0\}$,
    \item $\chi(ab) = \chi(a)\chi(b)$ for all $a, b \in \mathbb{F}_q \setminus \{0\}$.
\end{itemize}

The \emph{order} of $\chi$ is the smallest positive integer $e$ such that $\chi^e(a) = 1$ for all $a \in \mathbb{F}_q \setminus \{0\}$.
\end{definition}

\begin{theorem}[cf.~{\cite[Theorem 11.23]{zbMATH02121181}}]
Let $\chi$ be a multiplicative character of order $e$ on $\mathbb{F}_q$. Let $f \in \mathbb{F}_q[x]$ be a polynomial such that there does not exist $g \in \overline{\mathbb{F}}_q[x]$ with $f = g^e$. Then
\[
\left| \sum_{x \in \mathbb{F}_q} \chi(f(x)) \right| \leq (\deg{f}-1) q^{1/2}.
\]
\end{theorem}

\section{Conversion from Diophantine Tuples to Character Sums}
In this section, we transform the problem of counting the number of Diophantine tuples into character sums.

Let $f$ be an admissible polynomial in $\mathbb{F}_q[x]$ and 
$\mathscr{A}=\binom{[m]}{d}$. We consider the system:
\begin{equation}\label{eqf}
    \text{For all } A\in\mathscr A,\ \text{there exists } b_A\in\mathbb F_q 
    \text{ such that } 
    f\!\left(\prod_{i\in A} a_i\right)=b_A^2.
\end{equation}
Let $D$ be the set of all solutions to \eqref{eqf} in $\mathbb F_q$.  
Then the map
\[
\pi:D\to\mathbb F_q^{m},\qquad
\pi(a_1,\dots,a_m,b_{A_1},\dots,b_{A_{\binom{m}{d}}})=(a_1,\dots,a_m)
\]
is ``$2^{\binom{m}{d}}$-to-$1$''.%
\footnote{The quotation marks indicate that the mapping fails to be 
$2^{\binom{m}{d}}$-to-$1$ when some $b_A=0$, but such cases occur only 
$O(q^{m-1})$ times; see below.}
Indeed, given any $(a_1,\dots,a_m)\in\mathbb F_q^m$, if 
$(a_1,\dots,a_m,b_{A_1},\dots,b_{A_{\binom{m}{d}}})$ is a solution, then so 
is $(a_1,\dots,a_m,\pm b_{A_1},\dots,\pm b_{A_{\binom{m}{d}}})$.

Let $\mathrm{Conf}_m(\mathbb F_q)\subset\mathbb F_q^m$ be the subset of 
$m$-tuples $(a_1,\dots,a_m)\in\pi(D)$ satisfying $a_i\neq a_j$ for all 
$i<j$.

For each fixed $A\in\mathscr A$, the condition $b_A=0$ forces
$f\!\left(\prod_{i\in A} a_i\right)=0$, so $\prod_{i\in A} a_i$ must lie in 
the finite set $f^{-1}(0)$.  This gives a single algebraic constraint on 
$(a_1,\dots,a_m)$ and hence defines a hypersurface in $\mathbb F_q^m$ of 
size $O(q^{m-1})$.  A union bound over $A$ shows that the total number of 
solutions with $b_A=0$ for some $A$ is $O(q^{m-1})$.

Let $D^*=\pi^{-1}(\mathrm{Conf}_m(\mathbb F_q))$.  
Then
\[
N_f^{\binom{[m]}{d}}(q)
=
\frac{\mathrm{Card}(D^*)}{m!\,2^{\binom{m}{d}}}
+O(q^{m-1}),
\]
where the error term accounts for the solutions with $b_A=0$ for some $A$.
Since $\mathrm{Card}(\mathbb F_q^m\setminus\mathrm{Conf}_m(\mathbb F_q))
=O(q^{m-1})$. To prove Theorem~\ref{mainthe}, it suffices to show:
\begin{claim}\label{claim}
\[
\Card(D) = q^m +
\begin{cases}
O(q^{m-\frac{1}{2}}), & \deg(f) \ge 2,\\[4pt]
O(q^{m-1}), & \deg(f) = 1.
\end{cases}
\]
\end{claim}

To count the number of solutions to the system \eqref
{eqf}, we utilise the quadratic character $\chi$ on 
$\mathbb{F}_q$, the unique multiplicative character of 
order 2. For any $a \in \mathbb{F}_q$, the equation $x^2 = a$ has exactly $\chi(a) + 1$ solutions in $\mathbb{F}_q$.

Thus, the total number of solutions to \eqref{eqf} can be written as a product over all $A \in \mathscr{A}$ of terms of the form $\chi\left(f\left(\prod_{i \in A} a_i\right)\right) + 1$. Consequently,

\begin{align}
    \mathrm{Card}(D) &= \sum_{a_1, a_2, \ldots, a_m \in \mathbb{F}_q} \prod_{A \in \mathscr{A}} \left(1 + \chi\left(f\left(\prod_{i \in A} a_i\right)\right)\right) \nonumber \\
    &= q^m + \sum_{\varepsilon \in \{0,1\}^{\binom{m}{d}},\, \varepsilon \neq 0} R(\varepsilon) \label{eq:cardD}
\end{align}

where, for $\varepsilon = (\varepsilon_A)_{A \in \mathscr{A}} \in \{0,1\}^{\binom{m}{d}}$.
\begin{equation}\label{eqR}
    R(\varepsilon) = \sum_{a_1, a_2, \ldots, a_m \in \mathbb{F}_q} \prod_{A \in \mathscr{A}} \chi\left(f\left(\prod_{i \in A} a_i\right)\right)^{\varepsilon_A}
\end{equation}

\section{Evaluations on \(R(\varepsilon)\)}
In this section we divide the two cases $\wt(\varepsilon)= 1$ and $\wt(\varepsilon)\geq2$.
\begin{lemma}\label{lemma:epsilon1}
    If $\wt(\varepsilon)= 1$, i.e. only one component $\varepsilon_A=1$ and all others are zero, then
    \[
    R(\varepsilon) = \sum_{a_1,a_2,\ldots, a_m\in \mathbb{F}_q}\chi\!\left(f\!\left(\prod_{i\in A}a_i\right)\right) 
    \;\leq\; \left(\deg(f)-1\right) q^{m-\tfrac{1}{2}}+ (d-1) q^{m-1}
    \]
    where $\chi$ is the quadratic character on $\mathbb{F}_q$ and $f$ is a fixed admissible polynomial.
\end{lemma}

\begin{proof}
    In this case, since $\wt(\varepsilon) = 1$, there is a unique subset $A \subset [m]$ with $\Card(A)=d$ for which $\varepsilon_A = 1$. Thus, the sum in $R(\varepsilon)$ depends only on the $a_i$s for $i \in A$, while the remaining $a_j$s with $j \notin A$ contribute a factor of $q^{m-d}$.

    \begin{align}
            \sum_{a_1,a_2,\ldots, a_m\in \mathbb{F}_q}\chi\!\left(f\!\left(\prod_{i\in A}a_i\right)\right)
            &= q^{m-d}\cdot \sum_{a_i\in \mathbb{F}_q,\, i\in A} \chi\!\left(f\!\left(\prod_{i\in A}a_i\right)\right) \nonumber \\
            &= q^{m-d} \cdot \sum_{x\in\mathbb{F}_q} N(x)\, \chi\!\left(f(x)\right) \label{eq:sumNxchi}
    \end{align}
    where $N(x)$ denotes the number of $d$-tuples $(a_i)_{i\in A}\in (\mathbb{F}_q)^d$ satisfying $\prod\limits_{i\in A} a_i = x$. Clearly,
    \[
    N(x)=
    \begin{cases}
            (q-1)^{d-1}, & \text{if } x\neq 0,\\[6pt]
            q^d-(q-1)^{d}, & \text{if } x=0.
    \end{cases}
    \]
    
    Applying the Weil estimate, we obtain:
    \[
    \Biggl|\sum_{x\in\mathbb{F}_q} \chi\!\left(f(x)\right)\Biggr| \;\leq\; (\deg(f)-1)\, q^{\tfrac{1}{2}}.
    \]
    Using this in \eqref{eq:sumNxchi}
    \begin{align*}
        \sum_{a_1,\dots,a_m} \chi\left(f\left(\prod_{i\in A}a_i\right)\right) &= q^{m-d} \cdot (q-1)^{d-1} \cdot \sum_{x\in\mathbb{F}_q} \chi\left(f(x)\right) \\
        &\quad + q^{m-d} \cdot \left(q^d- q(q-1)^{d-1}\right) \cdot \chi \left(f(0)\right) \\
        &\leq (\deg(f)-1) q^{m-\tfrac{1}{2}} + (d-1) q^{m-1}\qedhere
    \end{align*}

\end{proof}

\begin{remark}
    When $\wt(\varepsilon)= 1$, the precise bound for $R(\varepsilon)$ mainly depends on the degree of the admissible polynomial $f$. If $\deg(f)=1$, then $R(\varepsilon) = O(q^{m-1})$, where the implied constant depends on $d$. If $\deg(f)>1$, then $R(\varepsilon) = O(q^{m-\frac{1}{2}})$, where the implied constant depends on $\deg(f)$.
\end{remark}

\begin{lemma}\label{lemma:epsilon-greater-1}
    Suppose $\wt(\varepsilon) = k > 1$, i.e., more than one component $\varepsilon_A = 1$, where $k$ denotes the number of subsets $A_1, \ldots, A_k \subset [m]$ of size $d$ with $\varepsilon_{A_j} = 1$. Then,
    \[
    R(\varepsilon) = \sum_{a_1,\dots,a_m\in \mathbb{F}_q} \prod_{j=1}^{k} \chi\left(f\left(\prod_{i\in A_j} a_i\right)\right) = O(q^{m-1}),
    \]
    where $\chi$ is the quadratic character on $\mathbb{F}_q$, $f$ is a fixed admissible polynomial, and $k$ is the number of nonzero entries in $\varepsilon$.
\end{lemma}

The following proof of Lemma~\ref{lemma:epsilon-greater-1} uses a method initially proposed by Shparlinski~\cite{zbMATH07697561}. When $\wt(\varepsilon)>1$, a direct application of the Weil bound yields only the trivial estimate $O(q^{m-1/2})$. The key idea is to introduce an auxiliary variable $b \in \mathbb{F}_q^*$ and exploit the multiplicative structure to separate the sum into two independent parts:
    \begin{enumerate}
        \item Find an index $a \in [m]$ that appears in some but not all of the sets $A_j$, partitioning them into groups $\mathscr{B}$ and $\mathscr{C}$.
        \item Summing over all $b\in\mathbb{F}_q^*$ and using the substitution
\(
(a_1,\dots,a_m)\longmapsto (a_1/b^{d-1},\, a_2 b,\,\dots,\, a_m b),
\)
which permutes $\mathbb{F}_q^m$, transforms the expression into a product of two independent character sums, one in $a_1$ and the other in $b$, where each yields a $q^{-1/2}$ power saving.
    \end{enumerate}
    This improves the error term from $O(q^{m-1/2})$ to $O(q^{m-1})$ by reducing the problem to two independent one-variable character sums.

\begin{theorem}[cf.~{\cite[I.~7, Prop. 1]{zbMATH00193787}}]
\label{theoremld}
	Let $K$ be an algebraically closed field and $V$ be an algebraic variety in $K^n$. Then $\dim V \leq n$. If $V \subsetneqq K^n$, then $\dim V < n$.
\end{theorem}

\begin{corollary}\label{corld}
    For a polynomial $g(l_1,\ldots , l_t;x)$ over $K=\overline{\mathbb{F}_q}$ with $t=m-1, l_i =a_{i+1}$ for $i=1,\ldots,t$. Let
    \[
    A=\{(l_1,\ldots,l_t)\in K^{t}\mid g(l_1,\ldots,l_t;x)\ \text{is a square polynomial over}\ K\}.
    \]
    If $\overline{A}\ne \varnothing$, then $\mathrm{Card}\bigl(A\bigr)=O(q^{t-1})$.
\end{corollary}
    
 \begin{proof}
    Suppose that $\deg g = e$, and write
    \[
    g(l_1,\ldots,l_t;x) = c_e x^e + c_{e-1} x^{e-1} + \ldots + c_0,
    \]
    where each coefficient \(c_i\) is a polynomial function of the parameters \(l_1,\ldots,l_t\), denoted by \(g_i(l_1,\ldots,l_t)\). As $(l_1,\ldots,l_t)$ varies, the tuples $(c_0,\ldots,c_e)$ form a constructible subset\footnote{A constructible subset of an algebraic variety is a finite union of locally closed subsets, i.e. subsets of the form $U \cap V^c$ where $U, V$ are Zariski closed.} of $K^{e+1}$~\cite[I.~8, Corollary~2]{zbMATH00193787}. Let $Z\subseteq K^{e+1}$ denote its Zariski closure; then $Z$ is an algebraic variety with $\dim Z \le t$.
    
    Consider the set of polynomials of degree at most\footnote{There may exist special values of $l_1,\ldots,l_t$ such that $c_e=0$.} $e$ that are squares. Any such polynomial can be written as
    \[
    h(x)^2 = g^*(x) = c_e^* x^e + \ldots + c_0^*,
    \]
    where $e' = \left\lfloor e/2 \right\rfloor$, so that $\deg h \le e'$, and
    \[
    h(x) = b_{e'} x^{e'} + \ldots + b_0
    \]
    with $b_i\in K$. Each coefficient $c_i^*$ is a polynomial function of $b_j$, so the tuples $(c_0^*,\ldots,c_e^*)$ form a constructible subset of $K^{e+1}$. Let $V\subseteq K^{e+1}$ denote its Zariski closure; then $V$ is an algebraic variety.
    
    Since $V \subseteq K^{e+1}$, it follows that $V \cap Z \subseteq Z$. Both $V$ and $Z$ are varieties in $K^{e+1}$, so $V \cap Z$ is also a variety in $K^{e+1}$. If $\overline{A} \neq \varnothing$, then $V \cap Z \subsetneqq Z$. By Theorem~\ref{theoremld}, we have $\dim(V \cap Z) < \dim(Z)$. Therefore,
    \[
    \dim(V \cap Z) \le t-1,
    \]
    which implies $\mathrm{Card}\bigl(A\bigr) = O(q^{t-1})$.\qedhere
\end{proof}

\begin{theorem}\label{thm:admissible-nonsquare}
    Let $f(x)\in\mathbb{F}_q[x]$ be an admissible polynomial. Then for every positive integer $k$ the polynomial $f(x^k)$ is not a perfect square in $\mathbb{F}_q[x]$.
    \end{theorem}
    \begin{proof}
        Let $\overline{\mathbb{F}_q}$ denote the algebraic closure of $\mathbb{F}_q$, and let $p$ be the characteristic of it. Factorize $f(x)$ in $\overline{\mathbb{F}_q}[x]$ according to its roots:
        $$
        f(x) = x^r \prod_{i=1}^m (x-\alpha_i)^{e_i},
        $$
        where $\alpha_i \text{s are distinct non-zero roots in } \overline{\mathbb{F}_q}^* $. By admissibility, the polynomial $h(x) = \prod_{i=1}^m (x-\alpha_i)^{e_i}$ satisfies $h(0) \neq 0$ and is not a perfect square in $\overline{\mathbb{F}_q}[x]$; hence at least one exponent $e_i$ is odd.

        For a positive integer \(k\geq 1\), write \(k=p^s t\) with \(s\ge 0\) and \(\gcd(t,p)=1\).  
    For each \(i\), the equation \(\beta^k=\alpha_i\) has exactly \(t\) distinct solutions in \(\overline{\mathbb{F}_q}\), and each solution occurs in \(x^k-\alpha_i\) with multiplicity \(p^s\).  
    Thus
    \[
    f(x^k)
    =
    x^{kr}
    \prod_{i=1}^m
    \prod_{\beta^k=\alpha_i}
    (x-\beta)^{e_i p^s}
    \quad \text{in } \overline{\mathbb{F}_q}[x].
    \]

    Since \(p\) is odd, so the multiplicity \(e_i p^s\) has the same parity as \(e_i\).
    Because some \(e_i\) is odd, the polynomial \(f(x^k)\) has at least one root of odd multiplicity.  
    Hence it cannot be a perfect square in \(\overline{\mathbb{F}_q}[x]\).
    \end{proof}

\begin{proof}[Proof of Lemma~\ref{lemma:epsilon-greater-1}]
    There exists an element $a$ in the union $A_1\cup\dots\cup A_k$ that does not belong to every $A_j$. Define
    \[
    \mathscr{B} = \{ j \mid a \in A_j \}, \qquad \mathscr{C} = \{ j \mid a \notin A_j \}.
    \]
    Both $\mathscr{B}$ and $\mathscr{C}$ are non-empty.
    
    Then we can write
    \[
    R(\varepsilon) = \sum_{a_1,\dots,a_m\in \mathbb{F}_q} \prod_{j\in\mathscr{B}} \chi\left(f\left(\prod_{i\in A_j} a_i\right)\right)
    \prod_{j\in\mathscr{C}} \chi\left(f\left(\prod_{i\in A_j} a_i\right)\right).
    \]
    
    Note that $\text{for all } b\in \mathbb{F}_q^{*}$, the map (without loss of generality, we assume $a = a_1$)
    \[
    (a_1,\dots,a_m) \mapsto \left(\frac{a_1}{b^{d-1}}, a_2 b, \dots, a_m b\right),
    \]
    is a bijection on $\mathbb{F}_q^m$. Hence,
    \[
    R(\varepsilon) = \frac{1}{q-1} \sum_{b\in \mathbb{F}_q^*} \sum_{a_1,\dots,a_m\in \mathbb{F}_q} 
    \prod_{j\in \mathscr{B}} \chi\Big(f\big(\prod_{i\in A_j} a_i\big)\Big)
    \prod_{j\in \mathscr{C}} \chi\Big(f\big(b^d \prod_{i\in A_j} a_i\big)\Big).
    \]
    
    Rearranging the sums over $a_1$ and $b$, we define
    \[
    F(a) = \prod_{j\in\mathscr{B}} f\Big(\prod_{i\in A_j} a_i\Big), \qquad
    G(b) = \prod_{j\in\mathscr{C}} f\Big(b^d \prod_{i\in A_j} a_i\Big),
    \]
    so that
    \[
    R(\varepsilon) = \frac{1}{q-1} \sum_{a_2,\dots,a_m\in \mathbb{F}_q} \left( \sum_{a_1\in \mathbb{F}_q} \chi(F(a)) \right) 
    \left( \sum_{b\in \mathbb{F}_q^*} \chi(G(b)) \right).
    \]
    
    Let
    \[
    \mathscr{D} = \left\{ (a_2, \dots, a_m) \mid F(a) \text{ is a square polynomial over } \overline{\mathbb{F}_q} \right\},
    \]
    \[
    \mathscr{E} = \left\{ (a_2, \dots, a_m) \mid G(b) \text{ is a square polynomial over } \overline{\mathbb{F}_q} \right\}.
    \]
    Note that when choosing a set $A_F\in\mathscr{B}$, let $a_i=1$ for all $i\in A_F\setminus\{1\}$, and $a_i=0$ for all $i\notin A_F$, we have 
    \[
    F(a) = f(a)\cdot f(0)^{\mathrm{Card}(\mathscr{B})-1},
    \]
    which is not a square polynomial over $\overline{\mathbb{F}_q}$. 

    Similarly, when choosing a set $A_G \in \mathscr{C}$, let $a_i = 1$, $\text{for all } i \in A_G$ and $a_i = 0$, $\text{for all } i \notin A_G$. Applying Theorem~\ref{thm:admissible-nonsquare}, we have
    \(
    G(b) = f(b^d) \cdot f(0)^{\mathrm{Card}(\mathscr{C})-1}
    \)
    which is not a square polynomial over $\overline{\mathbb{F}_q}$.
    
    By Corollary~\ref{corld}, we have $\mathrm{Card}(\mathscr{D}) = O(q^{m-2})$ and $\mathrm{Card}(\mathscr{E}) = O(q^{m-2})$.  Hence, $R(\varepsilon) = O(q^{m - 1})$.\qedhere
\end{proof}

\begin{proof}[Proof of Claim~\ref{claim}]
Applying Lemma~\ref{lemma:epsilon1}, and Lemma~\ref{lemma:epsilon-greater-1} to ~\eqref{eq:cardD}, we have
\begin{align*}
    \mathrm{Card}(D)&=q^m+\sum_{\varepsilon\in\{0,1\}^{\binom{m}{d}},\varepsilon\ne 0} R(\varepsilon)\\
    &=q^m+\sum_{\varepsilon\in\{0,1\}^{\binom{m}{d}},\wt{(\varepsilon)}=1} R(\varepsilon)+\sum_{\varepsilon\in\{0,1\}^{\binom{m}{d}},\wt{(\varepsilon)}>1} R(\varepsilon)\\
    &= 
    q^m +
    \begin{cases}
     O\left(q^{m -\frac{1}{2}}\right), & \text{if } \deg(f)\geq 2, \\
     O\left(q^{m -1}\right), & \text{if } \deg(f)=1
    \end{cases}
\end{align*}
Hence we have proven~\eqref{eq:cardD}, and this completes the proof of Theorem~\ref{mainthe}.\qedhere
\end{proof}

\begin{remark}
Extending this approach to more general symmetric polynomial conditions of the form \( f(a_{i_1}, a_{i_2}, \ldots, a_{i_d}) \) being a square, Shparlinski's technique fails, as the necessary separation of variables is no longer possible.

This limitation shows that the product form \( \prod\limits_{k=1}^d a_{i_k} \) is actually necessary for Shparlinski's method to work, since it enables the separation of variables and the application of bounds for character sums. The admissibility condition on $f$ further ensures the existence of examples where both $F$ and $G$ are not perfect squares, so that the non-emptiness condition in Corollary~\ref{corld} is satisfied.

For more general symmetric polynomial conditions, as considered by Kim, Yip, and Yoo \cite{Kim2025}, Shparlinski's approach is no longer directly applicable. Instead, they build on techniques from Slavov’s work\cite{2026Slavov}. Moreover, when $\deg(f)\geq2$, my result can be viewed as a special case of their Theorem~1.7.
\end{remark}

We end with two questions for further study.

\begin{enumerate}
\item Can one improve or adapt Shparlinski's method to more general symmetric polynomials?

\item The implied constant in the error term of~\eqref{mainthe} is not explicit. Can one compute it more precisely, as in \cite{Dujella}?
\end{enumerate}

\subsection*{Acknowledgements}

The author gratefully acknowledges Yaumathcamp and the Yauresearch Project for providing the opportunity to learn and conduct research through this project.

The author would like to thank Igor Shparlinski for inspiring discussions, for suggesting the problem studied in this paper, and for helping to correct the exposition of the article.

The author is also grateful to Andrej Dujella for his valuable comments, encouragement, and for pointing out errors in earlier versions of this paper.

The author would also like to thank Chi Hoi Yip for carefully reading the manuscript, pointing out errors, and for helpful communications regarding the problem.

\bibliographystyle{plain}
\bibliography{ref}

\end{document}